\documentclass[a4paper,12pt,reqno]{amsart}
\usepackage{amscd}
\usepackage{amsfonts}
\usepackage{mathrsfs}
\usepackage{amsmath,amssymb,latexsym,amsfonts,amscd}
\input xy
\xyoption{all}

\title{On canonically fibred algebraic 3-folds---some new examples}
\author{Meng Chen and Aoxiang Cui}

\address{\rm School of Mathematical Sciences \& LMNS, Fudan University,
Shanghai, 200433, China}
\email{mchen@fudan.edu.cn}

\address{\rm Department of Mathematics, Fudan University, Shanghai, 200433, China}
\email{081018001@fudan.edu.cn}

\thanks{Project
supported by the National Outstanding Young
Scientist Foundation (\#10625103) and NSFC key project (\#10731030)}

\newcommand{\bC}{{\mathbb C}}
\newcommand{\bQ}{{\mathbb Q}}
\newcommand{\bP}{{\mathbb P}}

\newcommand\OO{{\mathcal{O}}}

\newcommand\ZZ{{\mathbb{Z}}}

\newcommand{\D}{\displaystyle}

\newtheorem{thm}{Theorem}[section]
\newtheorem{lem}[thm]{Lemma}

\newtheorem{prop}[thm]{Proposition}

\theoremstyle{definition}

\newtheorem{setup}[thm]{}
\newtheorem{question}[thm]{Question}
\newtheorem{exmp}[thm]{Example}

\theoremstyle{remark}

\begin{document}
\numberwithin{equation}{section}
\maketitle
\pagestyle{myheadings} \markboth{\hfill M. Chen and A. X. Cui\hfill}{\hfill
On canonically fibred algebraic 3-folds \hfill}

\begin{abstract} This note aims to improve known numerical bounds proved earlier by Chen \cite{PAMS} and Chen-Hacon \cite{Chen-Hacon} and to present some new examples of smooth minimal 3-folds canonically fibred by surfaces (resp. curves) of geometric genus as large as $19$ (resp. $13$). As an interesting by-product, we present a new class of general type surfaces which are canonically fibred by curves of genus $13$.
\end{abstract}
\bigskip

\rightline{\it To the memory of Eckart Viehweg}

\section{\bf Introduction}

Let $V$ be a nonsingular projective 3-fold of general type. Assume $X$ is a minimal model of $V$ with at worst $\bQ$-factorial terminal singularities. When the geometric genus $p_g(X)\geq 2$, the canonical map $\varphi_1:=\Phi_{|K_X|}$ is usually a key tool for birational classification. If $\varphi_1$ is generically finite, Hacon \cite{Hacon} gave an example to show that $\deg(\varphi_1)$ is usually not bounded from above. If $\varphi_1$ is non-constant and not generically finite \mbox{(i.e. $0<\dim \overline{\varphi_1(X)}<3$)}, then $\varphi_1$ is said to be {\it of fiber type}. In this situation $V$ (or $X$) is said to be {\it canonically of fiber type}. Let $F$ be a birational smooth model of the generic irreducible component in the general fiber of $\varphi_1$. Denote by $g(F)$ (resp. $p_g(F)$) the genus (resp. the geometric genus) of $F$ when $F$ is a curve (resp. a surface). Clearly $F$ is of general type and nonsingular by simple addition formula and the Bertini theorem. We say that $X$ is {\it canonically fibred by curves $F$ (resp. surfaces $F$)}. It is interesting to see if the birational invariants of $F$ are bounded from above. Unfortunately such kind of boundedness was only proved when $X$ is Gorenstein (see Chen-Hacon \cite{Chen-Hacon}).

Restricting our interest to Gorenstein minimal 3-folds $X$ which are canonically of fiber type, Chen-Hacon \cite[Theorem 1.1]{Chen-Hacon} proved the desired boundedness theorem like $g(F)\leq 487$ (resp.  $p_g(F)\leq 434$). However, even if we assume $p_g(X)\gg 0$, the upper bounds (see, for instance, Chen \cite[Theorem 0.1]{PAMS}) for $g(F)$ (resp. $p_g(F)$)  might be far from optimal. Besides, among all known examples, the biggest value of $g(F)$ (resp. $p_g(F)$) is 5.

The motivation of this paper is to study the following:

\begin{question}\label{q} (cf. Chen-Hacon \cite[Question 4.2]{Chen-Hacon}) (1) For canonically fibred Gorenstein minimal 3-folds $X$ of general type, find optimal upper bounds of the invariants of fibers.

(2) Find new examples of $X$ such that the generic irreducible component in the general fiber of $\varphi_1$ has birational invariants as large as possible.
\end{question}

First of all, we aim at improving known upper bounds for $g(F)$ (resp. $p_g(F)$) and shall prove the
following theorem.

\begin{thm}\label{main}
Let $X$ be a Gorenstein minimal projective 3-fold of general type. Assume that $X$ is canonically of fiber type. Let $F$ be a smooth model of the generic irreducible component in the general
fiber of $\varphi_1$. Then
\begin{itemize}
\item[(i)] $g(F)\leq 91$ when $F$ is a curve and $p_g(X)\geq 183$;
\item[(ii)] $p_g(F)\leq 37$ when $F$ is a surface and $p_g(X)\geq 3890$.
\end{itemize}
\end{thm}

The second and the more important purpose of this paper is to present several new classes of 3-folds such as:
\begin{itemize}
\item[(1)] $X_{\text{S},19}$ (see Example \ref{S19}), canonically fibred by {\bf surfaces} of general type of geometric genus $p_g(F)=19$;
\item[(2)] $X_{\text{C}, \nu, 13}$ (see Example \ref{C13}), canonically fibred by {\bf curves} of genus $g(F)=13$, where $\nu\geq 3$;
\item[(3)] $Y_{\text{S}, 19}$ (see Example \ref{S19Y});
\item[(4)] $Y_{\text{C}, \nu, 13}$ (see Example \ref{C13Y}), $\nu\geq 3$;
\item[(5)] $Z_{\text{S}, 19}$ (see Example \ref{S19Z});
\item[(6)] $Z_{\text{C}, \nu, 13}$ (see Example \ref{C13Z}), $\nu\geq 3$;
\item[(7)] $X_{\text{S}, 16}$ (see Example \ref{S16X});
\item[(8)] $X_{\text{C}, \nu, 11}$ (see Example \ref{C11X}), $\nu\geq 3$;
\item[(9)] $X_{\text{S}, 13}$ (see Example \ref{S13});
\item[(10)] $X_{\text{C}, \nu, 9}$ (see Example \ref{C9}), $\nu\geq 3$.
 \end{itemize}

An earlier result due to Beauville \cite{Beauville} says that an algebraic surface $S$ of general type can be canonically fibred by curves of genus at most $5$, as long as $\chi(\OO_S)\geq 21$. However, there are no known examples of canonically fibred surfaces with fiber genus $>3$.  All above examples hint that, unlike the situation of surfaces, there might be plenty of canonically fibred minimal 3-folds.

A reward of above 3-fold structure is that we have actually found, in the last section, a new class of general type surfaces (see Example \ref{SS13}) which are canonically fibred by curves of genus $13$. 
\bigskip

As known to several experts, the main idea to treat this kind of questions is to deduce an effective inequality of Noether type while translating Miyaoka-Yau inequality (\cite{Miyaoka, Yau1, Yau2}) in terms of $p_g(X)$. We will go a little bit further in both sides in order to prove Theorem \ref{main}. In the authors' opinion, those examples in the last section are more interesting. Naively we feel that $p_g(F)\leq 37$ in Theorem \ref{main} (ii) is a nearly optimal upper bound.
\bigskip

Throughout we prefer to use ``$\sim$'' to denote linear equivalence whereas ``$\equiv$'' means numerical equivalence.
\bigskip

We would like to thank Fabrizio Catanese, Yongnam Lee, Margarida Mendes Lopes and Roberto Pignatelli for the consultant of relevant results on surfaces of general type with $p_g=0$. We thank Christopher D. Hacon for effective discussions through emails. Many thanks to the anonymous referee whose suggestion greatly improved the exposition of this note.

\section{\bf Technical preparation}

\begin{setup}\label{set}{\bf Set up for $\varphi_1$.}

Let $X$ be a Gorenstein minimal 3-fold of general type with $p_g(X):=h^0(X, \omega_X)\geq 2$. We denote by $\varphi_1$ the canonical rational map
$\varphi_1:X\dashrightarrow{\bP}^{p_g(X)-1}$. By Hironaka's theorem on resolution of singularities, we can take successive blow-ups $\pi:X'\rightarrow X$ along smooth centers, such that
\begin{itemize}
\item[(i)] $X'$ is nonsingular;
\item[(ii)] the movable part of $|K_{X'}|$ is base point free;
\item[(iii)] for a fixed Weil divisor $K_0\sim K_X$, the support of $\pi^*(K_0)+E$ is of simple normal crossing, where $E$ is the exceptional divisor of $\pi$ on $X'$.
\end{itemize}
Set $g:=\varphi_1\circ\pi$. Then \mbox{$g:X'\rightarrow W\subseteq\mathbb{P}^{p_g(X)-1}$}
is a morphism. Taking the Stein factorization of $g$, one gets $X'\overset{f}{\rightarrow} \Gamma \overset{s}\rightarrow W$ where $s$ is finite, $\Gamma$ is normal and $f$ is a fiber space.  So we have the following commutative diagram:
\begin{center}
\begin{picture}(50,80)(100,0) \put(100,0){$X$} \put(100,60){$X'$}
\put(170,0){$W$} \put(172,60){$\Gamma$}
\put(114,64){\vector(1,0){53}} \put(106,55){\vector(0,-1){41}}
\put(175,55){\vector(0,-1){43}} \put(114,58){\vector(1,-1){49}}
\multiput(112,2.6)(5,0){11}{-} \put(162,5){\vector(1,0){4}}
\put(133,70){$f$} \put(180,30){$s$} \put(95,30){$\pi$}
\put(133,-4){$\varphi_{1}$}\put(136,40){$g$}
\end{picture}
\end{center}
\medskip

The fibration $f:X'\rightarrow \Gamma$ is referred to as an {\it induced fibration} of $\varphi_1$.
Write $K_{X'}=\pi^*(K_X)+E$ and $|K_{X'}|=|M|+Z$
where $|M|$ is the movable part of $|K_{X'}|$, $E$ is exceptional and $Z$ the fixed part. Since $\pi^*(K_X)\geq M$, one may also write
$\pi^*(K_X)\sim M+E'$
where $E'$ is an effective divisor. In fact, $E'\leq Z$.

If $\dim\overline{\varphi_1(X)}\geq 2$, a general member $S$ of $|M|$ is a
nonsingular projective surface of general type.

If $\dim\overline{\varphi_1(X)}=1$, denote by $F$ a general fiber of $f$ and set $S=F$. Then $S$
is still a nonsingular projective surface of general type. Under this situation, one has
$$M=\sum_{i=1}^{a_1} F_i\equiv a_1 F.$$ where the $F_i'$s are smooth fibers
of $f$ and, clearly, $a_1\geq p_g(X)-1$.

In both cases, we call $S$ a \emph{generic irreducible
element} of $|M|$ or $|K_{X'}|$. Set
$$
p=
    \begin{cases}
        1 & \text{if $\dim(\Gamma)=2$},\\
        a_1 & \text{if $\dim(\Gamma)=1$}.
    \end{cases}
$$
Hence, we always have $M\equiv pS.$
\end{setup}

\begin{setup}\label{ineq1}{\bf Volume inequality}.

Assume $X$ is canonically of fiber type. Pick a generic irreducible element $S$ of $|K_{X'}|$. Suppose there is a movable linear system $|G|$ on $S$ with a smooth generic irreducible element $C$. Then, by Kodaira Lemma, there exists a positive rational number $\beta$ such that $\pi^*(K_X)|_S- \beta C$ is numerically equivalent to an effective $\bQ$-divisor. As being recognized in \cite[Inequality (2.1)]{Explicit2}, one has the following inequality:
\begin{equation}\label{vol} K_X^3\geq p\beta\xi\end{equation}
where $\xi:=(\pi^*(K_X).C)_{X'}$.
\end{setup}

\begin{setup}\label{ineq2}{\bf An inequality bounding $\xi$ from below}.

Keep the same notation as above. By \cite[Inequality (2.2)]{Explicit2}, one has
\begin{equation}\label{22}
\xi\geq \frac{2g(C)-2}{1+\D\frac{1}{p}+\D\frac{1}{\beta}}.\end{equation}
\end{setup}

\section{\bf The canonical family of curves}

Assume $\dim \Gamma=2$ in this section. We have an induced fibration $f:X'\rightarrow \Gamma$ with the general fiber a smooth curve $C$ with $g(C)\geq 2$. Keep the same notation as in \ref{set}.

\begin{prop}\label{g(C)}
Let $X$ be a Gorenstein minimal projective 3-fold of general type.
If $\dim\overline{\varphi_1(X)}=2$, then $$K_X^3\geq\Big\lceil\frac{2g(C)-2}{2+\frac{1}{p_g(X)-2}}\Big\rceil(p_g(X)-2).$$
In particular, $K_X^3\geq (g(C)-1)(p_g(X)-2)$ when $p_g(X)\geq 111$.
\end{prop}
\begin{proof} Take $|G|=|S|_S|$ on a general member $S$. Then $G\equiv \beta C$ with $\beta\geq p_g(X)-2$. Thus the inequality (\ref{vol}) gives
$$K^3_X\geq (\pi^*(K_X).C) (p_g(X)-2).$$
By inequality (\ref{22}), one has
$$\pi^*(K_X).C\geq \frac{2g(C)-2}{2+\frac{1}{p_g(X)-2}}.$$

The proof of Chen-Hacon \cite[Theorem 1.1(1)]{Chen-Hacon} actually implies $g(C)\leq 164$ whenever $p_g(X)\geq 111$. Now under the condition $p_g(X)\geq 111$, one gets $\pi^*(K_X).C>g(C)-2$. Since $\pi^*(K_X).C$ is an integer, we get $\pi^*(K_X).C\geq g(C)-1$.
\end{proof}

Combining Miyaoka-Yau inequality (Miyaoka \cite{Miyaoka} and Yau \cite{Yau1, Yau2}) and Theorem \ref{g(C)}, we only need to get an upper bound of $\chi(\omega_X)$.
In Chen-Hacon \cite[Proposition 2.1]{Chen-Hacon}, we know that
$$\chi(\omega_X)\le (1+\frac{1}{p_g(V_y)})p_g(V),$$
where $V$ is the smooth model of $X$, and $V_y$ is a generic irreducible component in the general fiber of Albanese morphism of $V$.

In the following theorem, we will bound $\chi(\omega_X)$ with a careful classification.

\medskip
\begin{thm}\label{2} Let $X$ be a Gorenstein minimal projective 3-fold of general type. Assume that $X$ is canonically fibred by curves $C$. Then $g(C)\leq 91$ whenever $p_g(X)\geq 183$.
\end{thm}
\begin{proof}

Assume $q(X)\leq 2$, then $\chi(\omega_X)\leq p_g(X)+1$ and one can easily verify the statement.

Assume $q(X)\geq 3$. Take a smooth birational model $V$ of $X$. Denote by $V_y$ a generic irreducible component
in the general fiber of the Albanese morphism of $V$. Note that $p_g(V_y)>0$ since $p_g(V)=p_g(X)>0$.

{\it Case (1)}. As long as $V_y$ is not a surface with $q(V_y)=0$ and $p_g(V_y)\leq 3$, then Chen-Hacon \cite[Proposition 2.1]{Chen-Hacon} implies $\chi(\omega_X)\leq \frac{5}{4}p_g(X)$.
Provided that $p_g(X)\geq 111$, one has:
$$(g(C)-1)(p_g(X)-2)\leq K_X^3\leq 72\cdot \frac{5}{4}p_g(X).$$
So, by calculation, one has $g(C)\leq 91$ when $p_g(X)\geq 183$.

{\it Case (2)}. If $V_y$ is a surface with $q(V_y)=0$ and $p_g(V_y)=1$, the canonical map $\Phi_{|K_V|}$ maps $V_y$ to a point, which contradicts to the assumption $\dim \Gamma=2$. Thus this is an impossible case.

{\it Case (3)}. If $V_y$ is a surface with $q(V_y)=0$ and $p_g(V_y)=2$, the assumption $\dim\Gamma=2$ implies that the natural restriction $j: H^0(V, K_V)\rightarrow H^0(V_y, K_{V_y})$ is surjective. This means $\Phi_{|K_V|}(V_y)=\mathbb{P}^1$. In other  words, ${\varphi_1}|_{V_y}=\Phi_{|K_{V_y}|}$. Taking further birational modification to $X'$ and $V$, we may assume $X'=V$ and the relative canonical map of $\text{Alb}_V$ is a morphism. Thus $C$, as a general fiber of $f$, is nothing but a generic irreducible element of the movable part of $|K_{V_y}|$. By the result of Beauville \cite{Beauville}, one easily gets, at worst, $g(C)\leq K_{V_y}^2+1\leq 9\chi(\OO_{V_y})+1\leq 28$.

{\it Case (4)}. If $V_y$ is a surface with $q(V_y)=0$ and $p_g(V_y)=3$, we will discuss it in details. In fact, $|K_{V_y}|$ can be either composed with a pencil of curves or not. We may still investigate the natural map $j$. Note that $\dim \text{Im}(j)\geq 2$. First, we consider the case $\dim\text{Im}(j)=2$. Then $|K_V||_{V_y}$ is a sub-pencil of $|K_{V_y}|$. Clearly the generic irreducible element of the movable part of $|K_V||_{V_y}$ is numerically equivalent to $C$, a general fiber of $f$ contained in $V_y$.  Denote by $V_{y,0}$ the minimal model of $V_y$. Noting that $\chi(\OO_{V_{y,0}})\leq 4$, one has
$$2g(C)-2\leq 2K_{V_{y,0}}^2\leq 18\chi(\OO_{V_{y,0}}),$$
which yields $g(C)\leq 37$. Next, we consider the case $\dim\text{Im}(j)=3$. This says $j$ is surjective. So $\Phi_{|K_V|}|_{V_y}=\Phi_{|K_{V_y}|}$. When $|K_{V_y}|$ is composed with a pencil, then the general fiber $C$ of $f$ contained in $V_y$ is nothing but a generic irreducible element in the movable part of $|K_{V,y}|$. By Beauville \cite{Beauville}, one has $g(C)\leq 36$. When $\dim\Phi_{|K_V|}(V_y)=2$, we hope to prove $\dim\varphi_1(V)=3$ which contradicts to the assumption $\dim\Gamma=2$. In fact, denote by $a:V\rightarrow Y$ the induced fibration after the Stein factorization of the Albanese map of $V$. Note that $g(Y)\geq q(V)\geq 3$. Replace $a$ by a relative minimal model $\hat{a}: Z\rightarrow Y$. Let $Z_{y}$ be a general fiber of $\hat{a}$. According to Fujita \cite{Fujita}, one knows $\hat{a}_*\omega_{Z/Y}$ is semi-positive. The Riemann-Roch on $Y$ gives
$$h^0(Y, \hat{a}_*\omega_Z\otimes \OO_Y(-y))\geq 3(g(Y)-2)>0.$$

Take a general point $y_1\in Y$. Since $|K_Z-Z_{y_1}|\neq\emptyset$, we may take a non-zero section $s_1\in H^0(Z, \omega_Z)$ such that $s_1$ vanishes along $Z_{y_1}$. Write $\text{div}{(s_1)}=Z_{y_1}+G_1$ with $G_1>0$. Pick another general fiber $Z_{y_2}$ such that $Z_{y_2}\not\leq Z_{y_1}+G_2$. Then $s_1$ does not vanish along $Z_{y_2}$. For the same reason, we may find another non-zero section $s_2\in H^0(Z, \omega_Z)$ such that $s_2$ vanishes along $Z_{y_2}$. This already implies
$\Phi_{|K_Z|}(Z_{y_1})\neq \Phi_{|K_Z|}(Z_{y_2})$. Thus $\dim \overline{\Phi_{|K_Z|}(Z)}>\dim \overline{\Phi_{|K_Z|}(Z_{y_1})}=2$, a contradiction. We have proved the theorem.
\end{proof}

\section{\bf The canonical family of surfaces}

In this section we assume $\dim\Gamma=1$. Essentially we will study the case when $p_g(F)$ is large. Recall that Chen \cite[Theorem 1]{JMSJ} implies $b=g(\Gamma)\leq 1$ whenever $p_g(F)>2$. {}First we will deduce a very delicate inequality of Noether type.

\begin{prop}\label{ep} Let $X$ be a Gorenstein minimal projective 3-fold of general type. Assume $X$ is canonically fibred by surfaces $F$ with $p_g(F)>2$. Then
{\small $$K_{X}^3\geq \begin{cases} (K_{F_0}^2+\frac{1}{4(20K_{F_0}^2+1)})p_g(X), &\ \ \ \text{when}\ \ b=1;\\
(K_{F_0}^2+\frac{1}{4(20K_{F_0}^2+1)})(p_g(X)-1)-\frac{4K_{F_0}^2}{2(20K_{F_0}^2+1)},&\begin{matrix}
&\text{when}\ b=0&\text{and}\\&p_g(X)\geq 56.& \end{matrix}
\end{cases}$$}
\end{prop}
\begin{proof} The property we are discussing here is birationally invariant. For technical reason, we need to choose a suitable minimal model. According to Kawamata \cite[Lemma 5.1]{Ka}, any Gorenstein minimal 3-fold is birational to a factorial Gorenstein minimal model with at worst terminal singularities. So we may assume that $X$ is factorial.  We keep the same notation as in \ref{set}. We have an induced fibration $f:X'\rightarrow \Gamma$ onto the smooth curve $\Gamma$.

Pick a general fiber $F$ of $f$. Set $N:=\pi_*(F)$ and
$Z:=\pi_*(E')$. Then we have $$K_X\equiv a_1N+Z.$$ As one knows
(see, for instance, Chen-Chen-Zhang \cite[2.2]{Crelle}), $K_X.N^2$
is a non-negative even integer.

In case $b=1$, automatically $K_X.N^2=0$ since the movable part of $|K_X|$ is base point free.

In case $b=0$ and $K_X.N^2>0$ (which means $K_X.N^2\geq 2$), one has
$$K_X^3\geq {a_1}K_X^2.N\geq 2{a_1}^2\geq 2(p_g(X)-1)^2.$$
Noting that $\chi(\omega_X)\leq \frac{3}{2}p_g(X)$ by Chen-Hacon \cite[2.2(2)]{Chen-Hacon}, the Miyaoka-Yau inequality implies $p_g(X)\leq 55$. In other words, $K_X.N^2=0$ whenever $p_g(X)\geq 56$.

Now we can work under the assumption $K_X.N^2=0$. Then one has $\pi^*(K_X)|_F=\sigma^*(K_{F_0})$ by Chen-Chen-Zhang \cite[Claim 3.3]{Crelle}. Therefore
$$K_X^2.N=\pi^*(K_X)^2.F=(\pi^*(K_X)|_F)^2=\sigma^*(K_{F_0})^2=K_{F_0}^2.$$

Recall that we are studying on the factorial minimal model $X$. According to Lee \cite{SLee}, $|4K_X|$ is base point free. Take a general member $S_4$
of $|4K_X|$.  Then $S_4$ is a nonsingular projective surface of general type. We hope to do some delicate calculation on $S_4$. Clearly $f(\pi^*(S_4))=\Gamma$, which also means that $S$ has a natural fibration structure since $(N|_{S_4})^2=4K_X.N^2=0$. Taking the restriction, one has:
$${K_X}|_{S_4}\equiv {a_1} N|_{S_4}+Z|_{S_4}.$$
The fact $(N|_{S_4})^2=0$ also implies that $f|_{\pi^*(S_4)}$ factors through $S_4$, i.e. there is a fibration $\nu:S_4\rightarrow \Gamma$. Since $S_4$ and $F$ are both general, one sees that $N|_{S_4}=\pi(F|_{\pi^*(S_4)})$ is irreducible and reduced and is exactly a general fiber of $\nu$.

By abuse of notation, we set $C:=N|_{S_4}$ and $G:=Z|_{S_4}$. Write $G:=G_v+G_h$ where $G_v$ is the vertical part while $G_h:=\sum_{i}m_iG_i$ ($m_i>0$) is the horizontal part. Then we have
$${K_X}|_{S_4}\equiv {a_1} C+G_v+\sum m_i G_i.$$
Note that $C$ is nef, $C^2=0$, $C.G_i>0$ for all $i$ and $p_a(G_i)\geq g(\Gamma)$. One has
\begin{eqnarray*}
4K_{F_0}^2&=&4(\pi^*(K_X)|_F)^2=4K_X^2.N\\
&=&\sum m_i(C.G_i)\geq \sum m_i.
\end{eqnarray*}
Write $k:=4K_{F_0}^2$. For each $i$, we have
$$K_{S_4}.G_i+G_i^2=2p_a(G_i)-2\geq 2b-2=2g(\Gamma)-2.$$

On $S_4$, take the divisor $D:=kK_{S_4}+2k(1-b)C+G_v+\sum m_i G_i$.
For each $i$, one has
\begin{eqnarray*}
D.G_i&\geq& k(K_{S_4}.G_i)+2k(1-b)(C.G_i)+m_i G_i^2\\
&\geq& m_i (K_{S_4}.G_i)+2m_i(1-b)(C.G_i)+m_i(2b-2-K_{S_4}.G_i)\\
&=& 2m_i((1-b)(C.G_i)+b-1)\geq 0.
\end{eqnarray*}
Thus $D.G_h\geq 0$. Explicitly,
$$D.G_h=k(K_{S_4}.G_h)+2k(1-b)(C.G_h)+G_v.G_h+G_h^2\geq 0.$$
Since
 $${K_X}|_{S_4}.G_h={a_1}(C.G_h)+G_v.G_h+G_h^2$$
and $K_{S_4}=(K_X+{S_4})|_{S_4}={5K_X}|_{S_4}$, by summing up together, one gets
$$(5k+1)({K_X}|_{S_4}.G_h)\geq ({a_1}-2k(1-b))(C.G_h)\geq {a_1}-2k(1-b).$$
Finally, one has
\begin{eqnarray*}
K_X^3&=& \frac{1}{4}(K_X|_{S_4})^2\geq \frac{1}{4}{a_1}(K_X|_{S_4}.N|_{S_4})+\frac{1}{4}(K_X|_{S_4}.G_h)\\
&\geq& (1+\frac{1}{k(5k+1)}){a_1}K_{F_0}^2-\frac{k}{2(5k+1)}(1-b)
\end{eqnarray*}
which yields the statements of this proposition.
\end{proof}

\begin{thm}\label{T2} Let $X$ be a Gorenstein minimal projective 3-fold of general type. Assume that $X$ is canonicaly fibred by surfaces. Let $f:X'\rightarrow \Gamma$ be an induced fibration and denote by $F$ a general fiber. Then $p_g(F)\leq 37$ when either $b=g(\Gamma)>0$ or $b=0$ and $p_g(X)\gg 0$, say $p_g(X)\geq 3890$.
\end{thm}
\begin{proof} If $b>1$, then $p_g(F)\leq 2$ by Chen \cite[Theorem 1]{JMSJ}.

If $b=1$, we know that $f_*\omega_{X'}=f_*\omega_{X'/\Gamma}$ and $R^1f_*\omega_{X'}=R^1f_*\omega_{X'/\Gamma}$
are both semi-positive (see Fujita \cite{Fujita}, Kawamata \cite{KaP}, Koll\'ar \cite{Kol}, Nakayama \cite{N} and Viehweg \cite{VP}). Furthermore $R^2f_*\omega_{X'}=\omega_{\Gamma}\cong \OO_{\Gamma}$.
This implies $\chi(R^1f_*\omega_{X'})\geq 0$ and $\chi(R^2f_*\omega_{X'})=0$.
Thus
$$\chi(\omega_{X'})=\chi(f_*\omega_{X'})-\chi(R^1f_*\omega_{X'})+\chi(R^2f_*\omega_{X'})\leq \chi(f_*\omega_{X'})\leq p_g(X).$$
So Proposition \ref{ep} and Miyaoka-Yau inequality imply
$$(K_{F_0}^2+\frac{1}{4(20K_{F_0}^2+1)})p_g(X)\leq 72p_g(X)$$
which directly gives $K_{F_0}^2<72$, whence $p_g(F)\leq 37$ by the Neother inequality $K_{F_0}^2\geq 2p_g(F)-4$.

If $b=0$, we assume $p_g(X)\geq 56$. Clearly
$$\chi(\omega_X)\leq p_g(X)+q(F)-1.$$
A preliminary estimation in Chen-Hacon \cite[2.2(2)]{Chen-Hacon} gives $K_{F_0}^2 \leq 108$ when $p_g(X)\geq 327$. We shall discuss by distinguishing the value of $q(F)$.

{\it Case (1)}. $q(F)>0$. By Debarre \cite{D}, one has $K_{F_0}^2\geq 2p_g(F)$. So Proposition \ref{ep} and Miyaoka-Yau inequality implies
$$K_{F_0}^2\leq 72-\frac{1}{4(20K_{F_0}^2+1)}+(36K_{F_0}^2+\frac{2K_{F_0}^2}{20K_{F_0}^2+1})\cdot \frac{1}{p_g(X)-1}.$$
The calculation shows $K_{F_0}^2\leq 72$ (whence $p_g(F)\leq 36$) if
$p_g(X)\geq 3890$; and $K_{F_0}^2\leq 71$ (whence $p_g(F)\leq 35$)
if $p_g(X)\geq 33616518$.

{\it Case (2)}. $q(F)=0$. Similarly, we have
$$K_{F_0}^2\leq 72-\frac{1}{4(20K_{F_0}^2+1)}+\frac{2K_{F_0}^2}{20K_{F_0}^2+1}\cdot \frac{1}{p_g(X)-1}.$$
At least, when $p_g(X)\geq 865$, one has $K_{F_0}^2<72$ or $p_g(F)\leq 37$.

To make the conclusion, when $p_g(X)\geq 3890$ and $b=0$, one gets $p_g(F)\leq 37$.
\end{proof}

Theorem \ref{2} and Theorem \ref{T2} imply Theorem \ref{main}.

\section{\bf New examples of canonically fibred 3-folds }

 An equally important task of birational classification is to provide supporting examples. One may refer to Chen \cite{PAMS} and Chen-Hacon \cite{Chen-Hacon} for some examples where the largest known value of $g(F)$ (resp. $p_g(F)$) is 5. In fact, it has been an open problem to look for general type 3-folds canonically fibred by curves (resp. surfaces) with invariant of general fiber as large as possible. In this section we would like to present some new examples.

\begin{setup}\label{standard}{\bf Standard construction}.

Let $S$ be a minimal projective surface of general type with $p_g(S)=0$. Assume there exists a divisor $H$ on $S$ such that $|K_S+H|$ is composed with a pencil of curves and that $2H$ is linearly equivalent to a smooth divisor $R$. The existence of such pair $(S,H)$ is secured by Lemma \ref{s}. Let $\hat{C}$ be a generic irreducible element of the movable part of $|K_S+H|$. Assume $\hat{C}$ is smooth. Set $d:=\hat{C}.H$ and $D:=\hat{C}\bigcap H$.

Let $C_0$ be a fixed smooth projective curve of genus 2. Let $\theta$ be a 2-torsion divisor on $C_0$. Set $Y:=S\times C_0$. Denote by $p_1:Y\rightarrow S$, $p_2:Y\rightarrow C_0$ the two projections. Take $\delta:=p_1^*(H)+p_2^*(\theta)$ and pick a smooth divisor $\Delta\sim p_1^*(2H)$. Then the pair $(\delta,\Delta)$ determines a smooth double covering $\pi:X\rightarrow Y$ and $K_X=\pi^*(K_Y+\delta)$. Clearly $X$ is smooth, minimal and of general type.

Since $K_Y+\delta=p_1^*(K_S+H)+p_2^*(K_{C_0}+\theta)$, $p_g(Y)=0$ and $h^0(K_{C_0}+\theta)=1$, one sees that $|K_X|=\pi^*|K_Y+\delta|$ and that $\Phi_{|K_X|}$ factors through $\pi$, $p_1$ and $\Phi_{|K_S+H|}$. Since $|K_S+H|$ is composed with a pencil of curves $\hat{C}$, $X$ is canonically fibred by surfaces $F$ and $F$ is a double covering over $T:=\hat{C}\times C_0$ corresponding to the data $(q_1^*(D)+q_2^*(\theta), q_1^*(2D))$ where $q_1$ and $q_2$ are projections. Denote by $\sigma:F\rightarrow T$ the double covering. Then $K_F=\sigma^*(K_T+q_1^*(D)+q_2^*(\theta))$. By calculation, one has $p_g(F)=3g(\hat{C})$ when $d=0$ and $p_g(F)=3g(\hat{C})+d-1$ whenever $d>0$.

The 3-folds constructed in the above way will be denoted by $X_{S, p_g(F)}$ in the context.
\end{setup}

\begin{setup}\label{v}{\bf Variant---An Infinite Family}.

In the construction \ref{standard}, if we replace $C_0$ by any smooth curve $C_{\nu}$ of genus $\nu\geq 3$, what we obtain is a smooth minimal 3-fold $X$ canonically fibred by curves. In fact, since $h^0(K_{C_0}+\theta)>1$, $\Phi_{|K_X|}$ factors through $\pi$ and $p_1\times p_2$. Thus a generic irreducible component in the fibres of $\Phi_{|K_X|}$ is simply a double covering $\tau: F\rightarrow \hat{C}$ branched along the divisor $2D$. Thus Hurwitz formula gives
$$2g(F)-2=2(2g(\hat{C})-2)+2d$$
and hence $g(F)=2g(\hat{C})+d-1$.

Clearly $p_g(X)$ can be arbitrarily large as long as $\nu=g(C_{\nu})$ is large. So such kind of 3-folds $X$ form an infinite family. We denote these 3-folds by $X_{C, \nu, g(F)}$.
\end{setup}

\begin{lem}\label{s} Let $S$ be any smooth minimal projective surface of general type with $p_g(S)=0$. Assume $\mu:S\rightarrow \bP^1$ is a genus 2 fibration. Let $H$ be a general fiber of $\mu$. Then $|K_S+H|$ is composed with a pencil of curves $\hat{C}$ of genus $g(\hat{C})$ and $\hat{C}.H=2$.
\end{lem}
\begin{proof} By Ramanujam's vanishing theorem \cite[P.131, Theorem 8.1]{BPV}, one has $H^1(S, K_S+H)=0$. Thus
$$h^0(S, K_S+H)=\frac{1}{2}(K_S+H)H+\chi(\OO_S)=g(H)=2$$
which gives that $|K_S+H|$ is naturally composed with a pencil of curves $\hat{C}$ of genus $g(\hat{C})$.

Since $q(S)\leq p_g(S)=0$, we have the surjective map
$$H^0(S,K_S+H)\rightarrow H^0(H, K_H).$$
This means that $|K_S+H|$ is composed with a different pencil from $|H|$. Thus $\mu(\hat{C})=\bP^1$. In other words, $|H||_{\hat{C}}$ is movable. The Riemann-Roch and the Clifford theorem simply imply $\hat{C}.H=2$ since $g(H)=2$.

{}Finally, whenever $K_S^2\geq 2$, Xiao \cite[Theorem 6.5]{X1137} proved that $S$ can not have two different pencils of genus 2 on $S$. We thus see $g(\hat{C})\geq 3$.
\end{proof}

We would like to look for those pairs $(S,H)$ satisfying the conditions of Lemma \ref{s}.

\begin{exmp}\label{S19}{\bf The 3-fold $X_{\text{S},19}$ which is canonically fibred by surfaces $F$ with $p_g(F)=19$.}

We take a pair $(S,H)$ which was found by Xiao \cite[P. 288]{X2}, where $S$ is a numerical Compedelli surface with $K_S^2=2$, $p_g(S)=q(S)=0$
and $\text{Tor}(S)=({\mathbb Z}_2)^3$. We need to recall the construction to determine the pencil $|\hat{C}|$ on $S$.

To start the construction, let $P=\bP^1\times \bP^1$. Denote by $x,\ y\in\bP^1=\bC\cup\{\infty\}$ the two coordinates of those points in $P$. Take four curves $C_1$, $C_2$, $C_3$ and $C_4$ defined by the following equations, respectively:
\begin{eqnarray*}
&C_1: &x=y;\\
&C_2: &x=-y;\\
&C_3: &xy=1;\\
&C_4: &xy=-1.
\end{eqnarray*}
These four curves intersect mutually at 12 ordinary double points:
\begin{eqnarray*}
& (0,0),\ (\infty,\infty),\ (0,\infty),\ (\infty,0)\\
&(\pm1,\pm1),\ (\pm \sqrt{-1}, \pm \sqrt{-1}).
\end{eqnarray*}

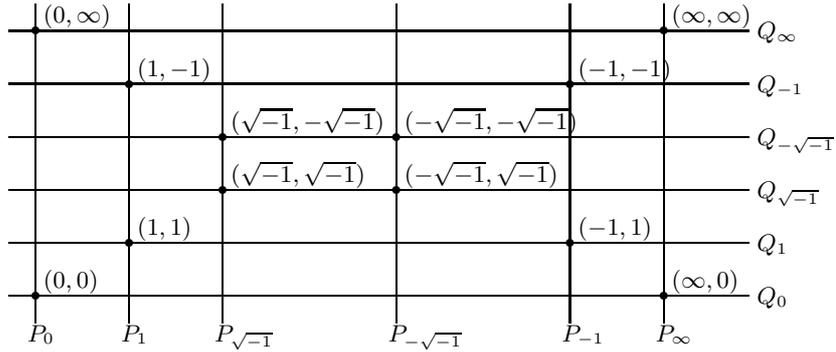
\begin{center}
\begin{picture}(290,140)
\put(10,0){\line(0,1){120}}
\put(45,0){\line(0,1){120}}
\put(80,0){\line(0,1){120}}
\put(145,0){\line(0,1){120}}
\put(210,0){\line(0,1){120}}
\put(245,0){\line(0,1){120}}

\put(0,10){\line(1,0){277}}
\put(0,30){\line(1,0){277}}
\put(0,50){\line(1,0){277}}
\put(0,70){\line(1,0){277}}
\put(0,90){\line(1,0){277}}
\put(0,110){\line(1,0){277}}

\put(10,10){\circle*{3}}
\put(13,13){\scriptsize{$(0,0)$}}
\put(10,110){\circle*{3}}
\put(13,113){\scriptsize{$(0,\infty)$}}

\put(45,30){\circle*{3}}
\put(48,33){\scriptsize{$(1,1)$}}
\put(45,90){\circle*{3}}
\put(48,93){\scriptsize{$(1,-1)$}}

\put(80,50){\circle*{3}}
\put(83,53){\scriptsize{$(\sqrt{-1},\sqrt{-1})$}}
\put(80,70){\circle*{3}}
\put(83,73){\scriptsize{$(\sqrt{-1},-\sqrt{-1})$}}

\put(145,50){\circle*{3}}
\put(148,53){\scriptsize{$(-\sqrt{-1},\sqrt{-1})$}}
\put(145,70){\circle*{3}}
\put(148,73){\scriptsize{$(-\sqrt{-1},-\sqrt{-1})$}}

\put(210,30){\circle*{3}}
\put(213,33){\scriptsize{$(-1,1)$}}
\put(210,90){\circle*{3}}
\put(213,93){\scriptsize{$(-1,-1)$}}

\put(245,10){\circle*{3}}
\put(248,13){\scriptsize{$(\infty,0)$}}
\put(245,110){\circle*{3}}
\put(248,113){\scriptsize{$(\infty,\infty)$}}

\put(7,-7){\scriptsize{$P_0$}}
\put(42,-7){\scriptsize{$P_1$}}
\put(77,-7){\scriptsize{$P_{\sqrt{-1}}$}}
\put(142,-7){\scriptsize{$P_{-\sqrt{-1}}$}}
\put(207,-7){\scriptsize{$P_{-1}$}}
\put(242,-7){\scriptsize{$P_\infty$}}

\put(280,7){\scriptsize{$Q_0$}}
\put(280,27){\scriptsize{$Q_1$}}
\put(280,47){\scriptsize{$Q_{\sqrt{-1}}$}}
\put(280,67){\scriptsize{$Q_{-\sqrt{-1}}$}}
\put(280,87){\scriptsize{$Q_{-1}$}}
\put(280,107){\scriptsize{$Q_{\infty}$}}
\end{picture}
\end{center}
\vskip0.4cm

Denote by $\varphi: P\rightarrow \bP^1$ be the first projection.
Let $F_1, \ldots, F_6$ be the fibers of $\varphi$ over $P_0$,
$P_{\infty}$, $P_1$, $P_{-1}$, $P_{\sqrt{-1}}$, $P_{-\sqrt{-1}}\in
\bP^1$, respectively. Note that each fiber $F_i$ contains exactly
two double points. Furthermore we take the following divisors of
bi-degree $(4,2)$:
\begin{eqnarray*}
&D_1=&C_1+C_2+F_1+F_2,\\
&D_2=&C_1+C_3+F_3+F_4,\\
&D_3=&C_2+C_3+F_5+F_6.
\end{eqnarray*}
Being linearly equivalent, passing through the above 12 points and having no common components, $D_1$, $D_2$ and $D_3$ generate a linear system which has the general smooth member $D$ with $D$ again passing through the 12 points. Let
\[R_1:=\sum_{i=1}^4 C_i+\sum_{i=1}^6 F_i+D\]
and $\delta_1$ be a divisor of bi-degree $(7,3)$ on $P$. Then the data $(\delta_1, R_1)$ determines a singular double covering onto $P$.
Noting that $R_1$ has exactly 12 singularities of multiplicity 4, we take the blowing up $\tau:\tilde{P}\rightarrow P$ resolving the 12 points. Then, on $\tilde{P}$, the strict transform $\tilde{R}_1:=\tau_*^{-1}(R_1)$ is smooth and the corresponding double covering $\theta:\tilde{S}\rightarrow \tilde{P}$ is smooth. Set $\tilde{f}:=\varphi\circ \tau\circ \theta$. Then we have a fibration $\tilde{f}: \tilde{S}\rightarrow \bP^1$. Clearly, on $\tilde{S}$, the strict transforms of $F_1$, $\ldots$, $F_6$ are exactly the only $(-1)$-curves and are contained in fibers of $\tilde{f}$. These $(-1)$-curves are contracted to get the relative minimal fibration $f:S\rightarrow \bP^1$ where $S$ happens to be minimal with $K_S^2=2$ and $p_g(S)=q(S)=0$. Denote by $\sigma:\tilde{S}\rightarrow S$ the blow down. Clearly $f$ is a fibration of genus $2$. Denote by $H$ a general fiber of $f$. We would like to study the linear system $|K_S+H|$.
\[
\begin{CD}
S @<\sigma<< \tilde{S} @>\theta>>\tilde{P}\\
@VfVV @VV\tilde{f}V @VV \tau V\\
\bP^1 @= \bP^1 @<<\varphi < P
\end{CD}
\]
Denote by $E_i=\theta_*^{-1}(\tilde{F}_i)$ and
$\tilde{F}_i:=\tau_*^{-1}(F_i)$ for $i=1,\ldots,6$. We have known that
$E_1, \ldots, E_6$ are all $(-1)$-curves. On $\tilde{P}$, let $J_1,
\ldots, J_{12}$ are 12 exceptional curves after the blowing up
$\tau$. Denote $L=\tau^*(1,0)$.  Then $\sigma^*(H)\sim \theta^*(L)$.

Take $\tilde{\delta}_1=\tau^*\delta_1-2\sum_{k=1}^{12}J_k$.
We have
\begin{eqnarray*}
K_{\tilde{S}}+\sigma^*(H)
&=&\theta^*(K_{\tilde{P}}+\tilde{\delta}_1+L)\\
&\sim & \theta^*(\tau^*(K_P+\delta_1+(1,0))-\sum_{k=1}^{12}J_k)\\
&\sim& \theta^*(\tau^*(6,1)-\sum_{k=1}^{12}J_k)\\
&\sim& \theta^*\tau^*(0,1)+2\sum_{i=1}^6 E_i.
\end{eqnarray*}
 Thus we see that $\sigma^*(K_S+H)\sim \theta^*\tau^*(0,1)+\sum_{i=1}^6 E_i$ and that $|K_S+H|$ has exactly 6 base points, but no fixed parts. Clearly a general member $\hat{C}\in |K_S+H|$ is obtained by mapping a general curve $ \theta^*\tau^*(0,1)$ onto $S$ and $\hat{C}$ is a smooth curve of genus $6$. Lemma \ref{s} tells that $|K_S+H|$ is composed with a pencil $\hat{C}$.

Now we take the triple $(S,H,\hat{C})$ and run Construction \ref{standard}. What we get is the 3-fold $X_{S, 19}$ which is canonically fibred by surfaces $F$ with $p_g(F)=19$. This is a new record with regard to \cite[Question 4.2]{Chen-Hacon}.
\end{exmp}

\begin{exmp}\label{C13}{\bf The 3-fold family $X_{\text{C}, \nu, 13}$ which are canonically fibred by curves $F$ of genus $g(F)=13$}.

We take the same triple $(S,H,\hat{C})$ as in Example \ref{S19} and run Construction \ref{v}. What we shall get is a 3-fold family $X_{C, \nu, 13}$ which are canonically fibred by curves $F$ of genus $13$ where $\nu\geq 3$.
This is again a new record with regard to \cite[Question 4.2]{Chen-Hacon}.
\end{exmp}

\begin{exmp}\label{S19Y}{\bf The 3-fold $Y_{\text{S},19}$.}

We take an alternative pair $(S,H)$ constructed by Weng (\cite{W}) and rephrased by Xiao as \cite[Examples 4.4.8]{Xiao_book}. Let $a,\ b\in \mathbb{C}$ be two numbers satisfying $a\neq 0, \pm 1$, $b\neq 0$, $ab\neq a+1$. Take $P=\bP^1\times \bP^1$. Denote by $x,\ y\in \bP^1=\bC\cup \{\infty\}$ the two coordinates of those points in $P$.
Take four projective curves defined by the following equations:
\begin{eqnarray*}
&C_1: &xy=1;\\
&C_2: &xy=a;\\
&C_3: &abx^3+ab(a+1)x+(a+1)^2y=(a+1)xy^2;\\
&C_4: &(a+1)y^2+(a+1)bx^2y^2+(a+1)bx^3y=abx^2.
\end{eqnarray*}
These curves have bi-degrees $(1,1)$, $(1,1)$, $(3,2)$ and $(3,2)$, respectively. The divisor $R_p:=C_1+C_2+C_3+C_4$ have the following triple points:
\begin{eqnarray*}
&(0,0), (0, \infty), (\infty, 0), (\infty, \infty);\\
& (\alpha_i, \D\frac{1}{\alpha_i}), (\alpha_i, \D\frac{a}{\alpha_i}), \quad i=1, \ldots, 4
\end{eqnarray*}
where the $\alpha_i'$s  ($i=1,\ldots,4$) are the roots of the following equation:
$$bx^4+(a+1)bx^2+(a+1)=0.$$
Note that $R_p$ has no more singular points. Take
$$R:=R_p+p_1^*(0+\infty+\alpha_1+\alpha_2+\alpha_3+\alpha_4).$$
For a divisor $\delta$ with $2\delta\sim R$, the pair $(\delta,R)$ gives a double covering which induces a minimal relative fibration $f:S\rightarrow \bP^1$ of genus $2$, where $S$ is a minimal surface of general type with $K_S^2=2$ and $p_g(S)=q(S)=0$. Let $H$ be a general fiber of $f$. Since $R$ has similar singularities as that in Example \ref{S19}, similar calculation shows that $|K_S+H|$ has exactly 6 base points, but no fixed parts. Thus a general member $\hat{C}\in |K_S+H|$ is a smooth curve of genus $6$. Lemma \ref{s} tells that $|K_S+H|$ is composed with a pencil $\hat{C}$.

Now if we take the triple $(S,H,\hat{C})$ and run Construction \ref{standard}. What we get is the 3-fold $Y_{S, 19}$ which is canonically fibred by surfaces $F$ with $p_g(F)=19$.
\end{exmp}

\begin{exmp}\label{C13Y}{\bf The 3-fold family $Y_{\text{C}, \nu, 13}$}.

Again we take the same triple $(S,H,\hat{C})$ as in Example \ref{S19Y} and run Construction \ref{v}. What we obtain is the 3-fold family $Y_{C, \nu, 13}$ which are canonically fibred by curves $F$ of genus $13$ with the parameter $\nu\geq 3$.
\end{exmp}

\begin{exmp}\label{S19Z}{\bf The 3-fold $Z_{\text{S},19}$.}

We take one more pair $(S,H)$ constructed by Weng and rephrased by Xiao as \cite[Examples 4.4.9]{Xiao_book}. Take $P=\bP^1\times \bP^1$. Denote by $x,\ y\in \bP^1=\bC\cup \{\infty\}$ the two coordinates of those points in $P$. Let $a\in \bC$ and $a\neq \pm1$, $\pm\sqrt{-1}$. For any point $c\in \bP^1=\bC\cup \{\infty\}$, $F_c$ denotes the fiber over $c$ of the fibration (first projection) $p_1:P\rightarrow \bP^1$. Consider the following five curves of bi-degree $(1,1)$:
\begin{eqnarray*}
&D_1: &x=y;\\
&D_2: &xy=1;\\
&D_3: &a^2y=x;\\
&D_4: &xy=a^2;\\
&D_5: &a^2xy=1.
\end{eqnarray*}
Take two curves of bi-degree $(5,3)$:
$$B_1=D_1+2D_5+F_a+F_{-a},$$
$$B_2=D_3+2D_4+F_1+F_{-1}.$$
One sees that $B_1$ and $B_2$ have no common components and they both pass through the following 12 points:
\begin{eqnarray*}
&(0,0),\ (\infty, \infty),\ (1, 1), &(-1, -1),\ (a,\D\frac{1}{a}),\ (-a, -\D\frac{1}{a});\\
& (0, \infty),\ (\infty,0),\ (1, \D\frac{1}{a^2}), &(-1, -\D\frac{1}{a^2}),\ (a,a), (-a, -a).
\end{eqnarray*}
Furthermore, $B_1$ and $B_2$ each has double points along:
$$(0,\infty), (\infty,0), (1,\frac{1}{a^2}), (-1, -\frac{1}{a^2}), (a,a), (-a, -a).$$
By Bertini theorem, the general member $B$ of the linear system generated by $B_1$ and $B_2$ is irreducible, $B$ passes through the above 12 points and the above mentioned 6 points are exactly the double points of $B$. Now take
$$R=B+D_1+D_2+D_3+F_0+F_{\infty}+F_1+F_{-1}+F_a+F_{-a}.$$
Then the 12 points are the only singular points of $R$ and each point is of multiplicity 4. Take a divisor $\delta$ such that $R\sim 2\delta$. Then $(\delta,R)$ again induces a minimal fibration $f:S\rightarrow \bP^1$ of genus 2, where $S$ is a minimal surface of general type with $K_S^2=2$ and $p_g(S)=q(S)=0$. Let $H$ be a general fiber of $f$. Since $R$ has similar singularities as that in Example \ref{S19}, similar calculation shows that $|K_S+H|$ has exactly 6 base points, but no fixed parts. Thus a general member $\hat{C}\in |K_S+H|$ is a smooth curve of genus $6$. Lemma \ref{s} tells that $|K_S+H|$ is composed with a pencil $\hat{C}$.

Now if we take the triple $(S,H,\hat{C})$ and run Construction \ref{standard}. What we get is the 3-fold $Z_{S, 19}$ which is canonically fibred by surfaces $F$ with $p_g(F)=19$.
\end{exmp}

\begin{exmp}\label{C13Z}{\bf The 3-fold family $Z_{\text{C}, \nu, 13}$}.

Still we take the same triple $(S,H,\hat{C})$ as in Example \ref{S19Z} and run Construction \ref{v}. What we obtain is a 3-fold family $Z_{C, \nu, 13}$ which are canonically fibred by curves $F$ of genus $13$.
\end{exmp}

\begin{exmp}\label{S16X}{\bf The 3-fold $X_{S,16}$}.

We take a pair $(S,H)$ found by Oort and Peters \cite{OP}, where $S$ is a numerical Godeaux surface
with $K_S^2=1$, $p_g(S)=q(S)=0$.
Consider four curves in $\mathbb{P}^2$ defined by the following equations:
\begin{align*}
C_1: \quad&Y^2+(X-1)(2X-3-2Y)=0;\\
C_2: \quad&Y^2+(X-1)(2X-3+2Y)=0;\\
C_3: \quad&Y^2+X(X-1)(X-3)=0;\\
C_4: \quad&(2-X)C_3+(X^2-3X+3)^2=0.
\end{align*}
where $C_1,C_2$ are quadratic curves and $C_3,C_4$ are cubic curves.
These four plane curves intersect at seven points:
\begin{gather*}
P_4=(\frac{3}{2},0), Q_1=(1,0), Q_2=(x_+,x_+), Q_3=(x_-,x_-),\\
Q_4=(x_+,-x_+), Q_5=(x_-,-x_-), P_2=(0,1).
\end{gather*}
where
$$x_+=\frac{3+\sqrt{-3}}{2}, \quad x_-=\frac{3-\sqrt{-3}}{2}.$$

\begin{center}
\begin{tabular}{c|c c c c c c c}
   & $P_4$ & $Q_1$ & $Q_2$ & $Q_3$ & $Q_4$ & $Q_5$ & $P_2$ \\
  \hline
  $C_1$ & 1 & 1 & 1 & 1 &   &   &   \\
  $C_2$ & 1 & 1 &   &   & 1 & 1 &   \\
  $C_3$ &   & 1 & 1 & 1 & 1 & 1 & 1 \\
  $C_4$ & 2 &   & 1 & 1 & 1 & 1 & 1 \\
\end{tabular}
\newline\newline
\end{center}

Let $D:=C_1+C_2+C_3+C_4$ which is a curve of degree 10. Since $P_4$
is an ordinary double point of $C_4$ and both $C_1,C_2$ pass through
$P_4$ with distinct tangent directions, $P_4$ is indeed an ordinary quadruple point of $D$. Triple points $Q_1,\ldots, Q_5$ are
of type $(3\rightarrow3)$(i.e. after single blowing ups, the triple points become ordinary triple points) and $P_2$ is an
ordinary double point of $D$.

Let $\delta:=\mathcal{O}_{\mathbb{P}^2}(5)$ be a divisor on $\mathbb{P}^2$. Covering data
$(\delta,D)$ determines a singular double covering over $\mathbb{P}^2$.
$D$ has five singular points of type $(3\rightarrow 3)$, one singular point of multiplicity 2 and one point of multiplicity 4.
Take successive blow-ups $\tau:\tilde{P}\rightarrow \mathbb{P}^2$ to resolve these seven singularities.

Denote by $E_{P_2}$, $E_{P_4}\in \tilde{P}$ the complete transform of the exceptional curves of $P_2$ and $P_4$, respectively.
The resolution for $Q_i$ needs two-step blow-ups, denote by $E_{Q_i}$, $E'_{Q_i}$ the two corresponding complete transforms of the exceptional curves.

Divisor $\tilde{D}:=\tau_*^{-1}(D)+\sum E_{Q_i}-\sum E'_{Q_i}$ on $\tilde{P}$ is smooth, hence the double covering
$\theta:\tilde{S}\rightarrow \tilde{P}$ is smooth. Denote by $\varphi:\mathbb{P}^2\dashrightarrow \mathbb{P}^1$ the rational map determined by the pencil of lines
passing through $P_4$ and let $\tilde{f}:=\varphi\circ\tau\circ\theta$. By Hurwitz formula we know that $\tilde{f}$
is a fibration of genus 2. On $\tilde{S}$, all the $(-1)$-curves are introduced by the desingularity of $Q_1,\ldots,Q_5$.
Contracting these five $(-1)$-curves gives a relative minimal fibration $f:S\rightarrow \mathbb{P}^1$ where $S$ is in fact
a minimal general type surface with $K_S^2=1$ and $p_g(S)=q(S)=0$. Clearly $f$ is also a fibration of genus 2. Denote by $H$ a general
fiber of $f$. Again we can determine the movable part of the linear system $|K_S+H|$.

$$\xymatrix{
S\ar[d]_{f} & \tilde{S}\ar[l]_{\sigma}\ar[r]^{\theta}\ar[d]^{\tilde{f}} & \tilde{P}\ar[d]^{\tau}\\
\mathbb{P}^1\ar@{=}[r]&\mathbb{P}^1&\mathbb{P}^2\ar@{-->}[l]^{\varphi}
}$$

Take
$$\tilde{\delta}=\tau^*\delta-E_{P_2}-2E_{P_4}-\sum E_{Q_i}-2\sum E'_{Q_i}.$$
We have
\begin{align*}
K_{\tilde{S}}+\sigma^*(H)&=\theta^*(K_{\tilde{P}}+\tilde{\delta})+\theta^*(\tau^*(\mathcal{O}_{\mathbb{P}^2}(1)-E_{P_4}))\\
&\sim \theta^*(\tau^*(\mathcal{O}_{\mathbb{P}^2}(3))-2E_{P_4}-\sum E'_{Q_i}).
\end{align*}

Denote by $E_i$ the five $(-1)$-curves on $\tilde{S}$, we have
\begin{align*}
\sigma^*(K_S+H)&=K_{\tilde{S}}+\sigma^*(H)-\sum{E_i}\\
&=\theta^*(\tau^*(\mathcal{O}_{\mathbb{P}^2}(3))-2E_{P_4}-\sum E_{Q_i})+\sum E_i.
\end{align*}

Now we want to find curves of degree 3 in $\mathbb{P}^2$ satisfying the following conditions:

(1) they pass through $P_4$ with multiplicity 2 and separated tangent directions;

(2) they pass through $Q_1, \ldots, Q_5$ with different tangent directions from that of $C_1, \ldots, C_4$ at these five points.\\
A direct computation shows that these curves form a linear system.
In fact, a general curve of the following form:
$$u(2(X-\frac{3}{2})^2Y-3(X-\frac{3}{2})Y+Y^3)+v(4(X-\frac{3}{2})^3+2(X-\frac{3}{2})^2+Y^2)=0$$
can be an appropriate candidate. So we have seen that $|K_S+H|$ has exactly 6 base points, but no fixed parts. Thus a general member $\hat{C}\in|K_S+H|$
is a smooth curve of genus $5$. Lemma \ref{s} implies that $|K_S+H|$ is composed with a pencil $\hat{C}$.

Now if we take the triple $(S,H,\hat{C})$ and run Construction \ref{standard}.
What we get is the 3-fold $X_{S, 16}$ which is canonically fibred by general type surfaces $F$ with $p_g(F)=16$.
\end{exmp}

\begin{exmp}\label{C11X}{\bf The 3-fold family $X_{C, \nu, 11}$}.

We take the same triple $(S,H,\hat{C})$ as in Example \ref{S16X} and run Construction \ref{v}.
What we obtain is the 3-fold family $X_{C, \nu, 11}$ which are canonically fibred by curves $F$ of genus $11$ with the parameter $\nu\geq 3$.
\end{exmp}

\begin{exmp}\label{S13}{\bf The 3-fold $X_{\text{S},13}$}.

Let $S$ be a minimal surface of general type with $K_S^2=1$ and $p_g(S)=0$. Take $H=K_S$. Then, among all known examples in Catanese-Pignatelli \cite{CP}, Lee \cite{Lee1, Lee2} and Reid \cite{Reid}, one knows that $|2K_S|$ is composed with a pencil of curves $\hat{C}$ of genus 4, $|2K_S|$ has no fixed part and a generic irreducible element $\hat{C}$ of $|2K_S|$ is smooth. Thus $g(\hat{C})=4$. Also one gets $d=\hat{C}.H=2K^2_S=2$. Take the triple $(S,H,\hat{C})$ and fill in Construction \ref{standard}.
Then we get the 3-fold $X_{S,13}$ which is canonically fibred by general type surfaces $F$ with $p_g(F)=13$.
\end{exmp}

\begin{exmp}\label{C9}{\bf The 3-fold family $X_{\text{C}, \nu, 9}$}.

Take the same triple $(S,H,\hat{C})$ as in Example \ref{S13} and run Construction \ref{v}, one gets the 3-fold family $X_{\text{C}, \nu, 9}$ which are canonically fibred by curves $F$ of genus $g(F)=2g(\hat{C})+d-1=9$.
\end{exmp}

{}Finally it is very interesting to know the answer to the following question:

\begin{question} Are there smooth (Gorenstein) minimal projective 3-folds of general type which are canonically fibred by surfaces (resp. curves) $F$ with $p_g(F)>19$ (resp. $g(F)>13$)?
\end{question}

\section{\bf A new class of canonically fibred surfaces of general type}

Let $M$ be a minimal projective surface of general type and assume that $|K_M|$ is composed with a pencil of curves of genus $g$. The existence of such surfaces with $p_g(M)\geq 3$ was known by Pompilij as early as 1984. In fact, there have been studies by Beauville \cite{Beauville}, Catanese \cite{Cat}, Debarre \cite{De}, Sun \cite{Sun}, Miyanishi-Yang \cite{M-Y}, Xiao \cite{X1} and others. Especially a minimal surface $M$ with $p_g=2$ is automatically canonically fibred by curves. However, among all known canonically fibred surfaces, very few examples with $g>3$ are known in literature. 

Inspired by our construction in the last section. We are able to present here at least 3 new examples with $g=13$. Of course, our construction below has the potential to illustrate other examples with $g>3$.  

\begin{exmp}\label{SS13} Take a pair $(S,H)$ satisfying Lemma \ref{s}. Let $\zeta: \hat{S}\longrightarrow S$ be the double cover corresponding to the datum $(\delta, \Delta)=(H, H_1+H_2)$ with $H_1\sim H$, $H_2\sim H$. Since $\Delta$ is smooth, we see that $\hat{S}$ is a minimal surface of general type with $K_{\hat{S}}^2=2(K_S+H)^2$ and $p_g(\hat{S})=h^0(S, K_S+H)=2$. Thus $|K_{\hat{S}}|$ is automatically composed with a pencil of curves of genus $g$.
\medskip

\noindent{\bf (\ref{SS13}.1) The surface $\hat{S}_1$ with $g=13$}. If we take the same pair $(S,H)$ as in Example \ref{S19}, what we get is a new surface $\hat{S}_1$ with $K_{\hat{S}_1}^2=12$ and, corresponding to the structure of Example \ref{C13}, $|K_{\hat{S}_1}|$ is composed with a
pencil of curves of genus $g=13$, since a general member $\hat{C}\in |K_S+H|$ is smooth as proved in Example \ref{S19}.

\noindent{\bf (\ref{SS13}.2) The surface $\hat{S}_2$ with $g=13$}. Similarly if we take the same pair $(S,H)$ as in Example \ref{S19Y}, what we get is another new surface $\hat{S}_2$ with $K^2=12$ and $g=13$. This example is corresponding to Example \ref{C13Y}.

\noindent{\bf (\ref{SS13}.3) The surface $\hat{S}_3$ with $g=13$}. If we take the same pair $(S,H)$ as in Example \ref{S19Z}, what we get is again a new surface $\hat{S}_3$ with $K^2=12$ and $g=13$. This example is corresponding to Example \ref{C13Z}.
\medskip

Anyway $g=13$ is the largest possibility that we can hope to obtain no matter what the pair $(S,H)$ is.
\end{exmp}

Finally we would like to ask the following:

\begin{question} Can one find canonically fibred general type surfaces with $g>13$?
\end{question}

\end{document}